\documentclass[leqno]{amsart}
\usepackage{geometry}
\usepackage{caption}
\usepackage{floatrow}
\usepackage{enumitem}
\usepackage{mathrsfs}

\newcommand{\cT}{\mathcal{T}}
\newcommand{\eps}{\varepsilon}
\newcommand{\friendly}{friendly}
\newcommand{\Friendly}{Friendly}

\usepackage{epsfig}
\usepackage{graphicx}
\usepackage{amsbsy}
\usepackage{amsmath}
\usepackage{amsfonts}
\usepackage{amssymb}
\usepackage{textcomp}
\usepackage{hyperref}
\usepackage{aliascnt}

\newcommand{\mcm}[3]{\newcommand{#1}[#2]{{\ensuremath{#3}}}} 

\mcm{\tuple}{1}{\langle #1 \rangle}
\mcm{\name}{1}{\ulcorner #1 \urcorner}
\mcm{\Nbb}{0}{\mathbb{N}}
\mcm{\Zbb}{0}{\mathbb{Z}}
\mcm{\Rbb}{0}{\mathbb{R}}
\mcm{\Cbb}{0}{\mathbb{C}}
\mcm{\Qbb}{0}{\mathbb{Q}}
\mcm{\Acal}{0}{\cal A}
\mcm{\Bcal}{0}{\cal B}
\mcm{\Ccal}{0}{\cal C}
\mcm{\Dcal}{0}{\cal D}
\mcm{\Ecal}{0}{\mathcal E}
\mcm{\Fcal}{0}{\cal F}
\mcm{\Gcal}{0}{\cal G}
\mcm{\Hcal}{0}{\cal H}
\mcm{\Ical}{0}{\cal I}
\mcm{\Jcal}{0}{\cal J}
\mcm{\Kcal}{0}{\cal K}
\mcm{\Lcal}{0}{\cal L}
\mcm{\Mcal}{0}{\cal M}
\mcm{\Ncal}{0}{\cal N}
\mcm{\Ocal}{0}{{\cal O}}
\mcm{\Pcal}{0}{{\cal P}}
\mcm{\Qcal}{0}{{\cal Q}}
\mcm{\Rcal}{0}{{\cal R}}
\mcm{\Scal}{0}{{\cal S}}
\mcm{\Tcal}{0}{{\mathcal T}}
\mcm{\Ucal}{0}{{\cal U}}
\mcm{\Vcal}{0}{{\mathcal V}}
\mcm{\Wcal}{0}{{\cal W}}
\mcm{\Xcal}{0}{{\cal X}}
\mcm{\Ycal}{0}{{\cal Y}}
\mcm{\Zcal}{0}{{\cal Z}}
\mcm{\Mfrak}{0}{\mathfrak M}

\mcm{\restric}{0}{\upharpoonright}
\mcm{\upset}{0}{\uparrow}
\mcm{\onto}{0}{\twoheadrightarrow}
\mcm{\smallNbb}{0}{{\small \mathbb{N}}}
\DeclareMathOperator{\preop}{op}
\mcm{\op}{0}{^{\preop}}

\newcommand{\se}{\subseteq}
{\begin{array}{c}
\setlength{\unitlength}{1em}}%
{\end{array}}

\usepackage{amsthm}

\newcommand{\theoremize}[2]{\newaliascnt{#1}{thm} \newtheorem{#1}[#1]{#2} \aliascntresetthe{#1}}

\theoremstyle{plain}
\newtheorem{thm}{Theorem}[section]
\theoremize{lem}{Lemma}
\theoremize{skolem}{Skolem}
\theoremize{fact}{Fact}
\theoremize{sublem}{Sublemma}
\theoremize{claim}{Claim}
\theoremize{obs}{Observation}
\theoremize{prop}{Proposition}
\theoremize{cor}{Corollary}
\theoremize{que}{Question}
\theoremize{oque}{Open Question}
\theoremize{con}{Conjecture}

\theoremstyle{definition}
\theoremize{dfn}{Definition}
\newtheorem*{maindfn}{Definition}
\theoremize{rem}{Remark}
\theoremize{eg}{Example}
\theoremize{exercise}{Exercise}
\theoremstyle{plain}

\mcm{\Fbb}{0}{\mathbb{F}}
\newcommand{\sm}{\setminus}

\theoremstyle{plain}
\newtheorem{main}{Theorem}
\theoremize{prob}{Problem}

\usepackage{xcolor}
\colorlet{darkishRed}{red!80!black}
\colorlet{darkishBlue}{blue!60!black}
\colorlet{darkishGreen}{green!60!black}
\hypersetup{
    draft = false,
    bookmarksopen=true,
    colorlinks,
    linkcolor={red!60!black},
    citecolor={green!60!black},
    urlcolor={blue!60!black}
}

\renewcommand{\le}{\leqslant}
\renewcommand{\ge}{\geqslant}

\begin{document}
\title{\vspace*{-2.5cm}Entanglements}
\author[Johannes Carmesin]{Johannes Carmesin${}^\clubsuit$}
\address{University of Birmingham, Birmingham, UK}
\email{j.carmesin@bham.ac.uk, j.kurkofka@bham.ac.uk}
\author[Jan Kurkofka]{Jan Kurkofka${}^\clubsuit$}
\thanks{${}^\clubsuit$University of Birmingham, Birmingham, UK, funded by EPSRC, grant number EP/T016221/1}
\keywords{entanglement, tree of tangles, nested set of separations, efficiently distinguish, canonical}
\subjclass[2020]{05C83, 05C40, 05C05}
\begin{abstract}
\noindent 
Robertson and Seymour constructed for every graph $G$ a tree-decomposition that efficiently distinguishes all the tangles in~$G$. 
While all previous constructions of these decompositions are either iterative in nature or not canonical, we give an explicit one-step construction that is canonical.

The key ingredient is an axiomatisation of \lq local properties\rq\ of tangles.
Generalisations to locally finite graphs and matroids are also discussed.
\end{abstract}
\maketitle

\vspace{-.7cm}
\section{Introduction}

In this paper we propose an axiomatisation of \lq local properties\rq\ of tangles and apply it to give explicit one-step constructions of tree-decompositions, as follows.

Roughly speaking, \emph{tree-decompositions} are a recipe how to cut up a graph along separations in a tree-like way. \emph{Tangles} are a way to axiomatise highly cohesive substructures in graphs such as complete subgraphs or grid minors. 
We say that a separation $\{A,B\}$ of a graph $G$ \emph{distinguishes} a pair of tangles  if the two tangles live on opposite sides of $\{A,B\}$; it does so \emph{efficiently} if the separator of $\{A,B\}$ has smallest size amongst all distinguishing separations of $G$.
We say that a tree-decomposition of a graph $G$ (\emph{efficiently}) \emph{distinguishes} a pair of tangles  if there is a separation  $\{A,B\}$ which (efficiently) distinguishes the two tangles and $\{A,B\}$ is in the recipe for the tree-decomposition.
A key tool~\cite{GMX} in the proof of the graph-minor theorem states:
\begin{equation}\label{OriginalToT}
\textit{Every finite graph }G\textit{ has a tree-decomposition that efficiently distinguishes all the tangles in }G.
\end{equation}
A fair amount of the recent work on graph-minors has focused on constructing such tree-decompositions~\cite{CarmesinToTshort,CDHH13CanonicalAlg,CDHH13CanonicalParts,confing,CG14:isolatingblocks,DiestelBook16noEE,ProfilesNew,InfiniteSplinters,jacobs2023efficiently}.
In all proofs in the literature these tree-decompositions are constructed through an iterative process in which separations are chosen in turn based on previous choices. 
Here we will give a new construction of the tree-decomposition of (\ref{OriginalToT}) that finishes in one step, is canonical, and that is \emph{explicit} in the sense that it computes a single simple parameter for separations and then takes all separations for the tree-decomposition which minimise this parameter. 

In the proof of (\ref{OriginalToT}), one has to construct separations that disinguish all pairs of tangles efficiently, and one has to construct them in a \emph{nested} way; that is, so that they define the recipe of a tree-decomposition. 
Rather than working with tangles in the first place, our perspective is to directly axiomatise separations which distinguish tangles efficiently through a new notion of \emph{entanglements}; see \autoref{sec:Entangle}. 
Perhaps surprisingly, these entanglements have very similar properties to tangles themselves but only applied to a subset of their separations. See \autoref{sec:Entangle} for an explanation of why we think of entanglements as an axiomatisation of `local properties' of tangles.

Our main result reads as follows.

\begin{maindfn}[Friendly]
A separation $\{A,B\}$ in an entanglement $\eps$ in~$G$ is \emph{\friendly } if no other separation in~$\eps$ crosses less separations in entanglements in~$G$ than~$\{A,B\}$.
\end{maindfn}

\begin{main}\label{Main}
For every finite graph~$G$, the set of \friendly\ separations of~$G$ is a nested set of separations; and hence gives rise to a tree-decomposition distinguishing all tangles efficiently.
\end{main}

The nested sets $N(G)$ and tree-decompositions $\cT (G)$ provided by \autoref{Main} are canonical in that they commute with graph-isomorphisms: $\varphi(N(G))=N(\varphi(G))$ and $\varphi(\cT(G))=\cT(\varphi(G))$ for every graph-isomorphism $\varphi\colon G\to G'$.

The decomposition in~\autoref{Main} refines the one of~(\ref{OriginalToT}).
Indeed, not every entanglement is induced by a pair of tangles, and in fact entanglements and friendly separations can be found in graphs that host no tangles at all (\autoref{FareyExample}).

\autoref{Main} extends to locally-finite infinite graphs under additional assumptions; see \autoref{TechnicalMain}.
We also provide an abstract version of \autoref{Main}, inspired by~\cite{ASS,TreeSets,ProfilesNew,FiniteSplinters}, which can be applied to a wide variety of setups including matroids;
see \autoref{sec:AbstractEntangle}.

This note is organised as follows.
Entanglements in graphs are introduced in \autoref{sec:Entangle}.
\autoref{Main} is proved in \autoref{sec:OptimalNested}.
An infinite version of \autoref{Main} is proved in \autoref{sec:Infinite}.
Abstract versions of entanglements and of \autoref{Main} are offered in \autoref{sec:AbstractEntangle}.

\section{Entanglements in graphs}\label{sec:Entangle}

Let $G$ be any graph.
A \emph{separation} of~$G$ is a set $\{A,B\}$ such that $A\cup B=V(G)$ and $G$ contains no edge between $A\sm B$ and $B\sm A$.
We refer to $A$ and $B$ as the \emph{sides} of~$\{A,B\}$, and call $A\cap B$ the \emph{separator} of~$\{A,B\}$.
The size $\vert A\cap B\vert$ of the separator is the \emph{order} of~$\{A,B\}$.
A~separation~$\{A,B\}$ is \emph{proper} if $A\sm B$ and $B\sm A$ are non-empty.
Two separations $\{A,B\}$ and $\{C,D\}$ of~$G$ are \emph{nested}~if, after possibly renaming their sides, they satisfy $A\se C$ and $B\supseteq D$.
Two separations that are not nested are said to \emph{cross}.
A set of separations of~$G$ is \emph{nested} if its elements are pairwise nested.

\begin{figure}[ht]
    \centering
    \includegraphics[height=5\baselineskip]{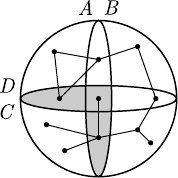}
    \caption{$\{A\cap C, B\cup D\}$ is one of the four corners of $\{A,B\}$ and~$\{C,D\}$}
    \label{fig:CrossingDiagram}
\end{figure}

For a depiction of the setting for the next definitions, see \autoref{fig:CrossingDiagram}.
If $\{A,B\}$ and $\{C,D\}$ cross, then their four \emph{corners} are the separations $\{A\cap C,B\cup D\}$, $\{A\cap D,B\cup C\}$, $\{B\cap D,A\cup C\}$ and $\{B\cap C,A\cup D\}$. 
The corners $\{A\cap C,B\cup D\}$ and $\{B\cap D,A\cup C\}$ are \emph{opposite}, and so are the corners $\{A\cap D,B\cup C\}$ and~$\{B\cap C,A\cup D\}$.
Any two corners that are not opposite are \emph{adjacent}.
The two adjacent corners $\{A\cap C,B\cup D\}$ and $\{A\cap D,B\cup C\}$ are said to \emph{lie on the same side} of~$\{A,B\}$.
Similarly, the two adjacent corners $\{B\cap D,A\cup C\}$ and $\{B\cap C,A\cup D\}$ are said to \emph{lie on the same side} of~$\{A,B\}$.

An \emph{entanglement} in~$G$ is a non-empty set $\eps$ of proper separations of~$G$ such that~$\eps$ satisfies~\ref{Entangle1}:
\begin{enumerate}[label=($\Ecal$)]
    \item\label{Entangle1} 
    If a separation~$\{A,B\}\in \eps$ is crossed by a separation of~$G$ so that two corners lying on the same side of~$\{A,B\}$ have order at most~$\vert A\cap B\vert$, then at least one of these corners has order equal to~$\vert A\cap B\vert$ and is contained in~$\eps$.
\end{enumerate}

A separation $\{A,B\}$ in an entanglement $\eps$ in~$G$ is \emph{\friendly } if no other separation in~$\eps$ crosses less separations in entanglements in~$G$ than~$\{A,B\}$.

We conclude this section with three examples.
The first example uses the terminology of~\cite[§12.5]{DiestelBook16noEE}. 
We state it as a lemma because it is a key ingredient of the proof of \autoref{Main}.

\begin{lem}
\label{TanglesInduceEntanglements}
Every pair of distinguishable tangles in a graph induces an entanglement, which consists of the separations efficiently distinguishing the two tangles.
\end{lem}
\begin{proof}
Let $\tau$ and $\tau'$ be two distinguishable tangles in a graph~$G$, and let $\eps$ be the set of all separations of~$G$ which efficiently distinguish $\tau$ and~$\tau'$.
The set $\eps$ is non-empty since $\tau$ and~$\tau'$ are distinguishable, and the separations in~$\eps$ are proper because tangles do not contain separations of the form~$(V(G),B)$.
We claim that $\eps$ satisfies~\ref{Entangle1}.
For this, suppose that $\{A,B\}\in \eps$ is crossed by a separation $\{C,D\}$ of~$G$ so that the two corners $c_1:=\{A\cap C,B\cup D\}$ and $c_2:=\{A\cap D,B\cup C\}$ have order at most~$\vert A\cap B\vert$.
Without loss of generality, $\tau$ orients $\{A,B\}$ towards~$A$ and $\tau'$ orients $\{A,B\}$ towards~$B$.
Since the corners $c_1$ and $c_2$ have order at most~$\vert A\cap B\vert$, they are oriented by~$\tau$ and~$\tau'$.
The tangle $\tau'$ orients both $c_1$ and $c_2$ towards $B$ by consistency.
The tangle $\tau$ cannot orient both $c_1$ and $c_2$ towards $B$ since tangles do not contain three separations whose small sides together cover~$G$.
Therefore, $\tau$~orients one of $c_1$ and $c_2$ away from~$B$.
Then that corner distinguishes $\tau$ and~$\tau'$, and must do so efficiently, hence it lies in~$\eps$.
\end{proof}

If $\tau$ is a tangle in~$G$, and $\sigma_i$ for $i\in I$ are the tangles in~$G$ that are distinguishable from~$\tau$, then for every $\sigma_i$ we obtain an entanglement~$\eps_i\se\tau$ by \autoref{TanglesInduceEntanglements}, and these $\eps_i$ contain all the information from~$\tau$ that is sufficient to efficiently distinguish $\tau$ from all $\sigma_i$. This is why intuitively, we may think of entanglements as an axiomatisation of `local properties' of tangles.

\begin{figure}[ht]
    \centering
    \includegraphics[height=5\baselineskip]{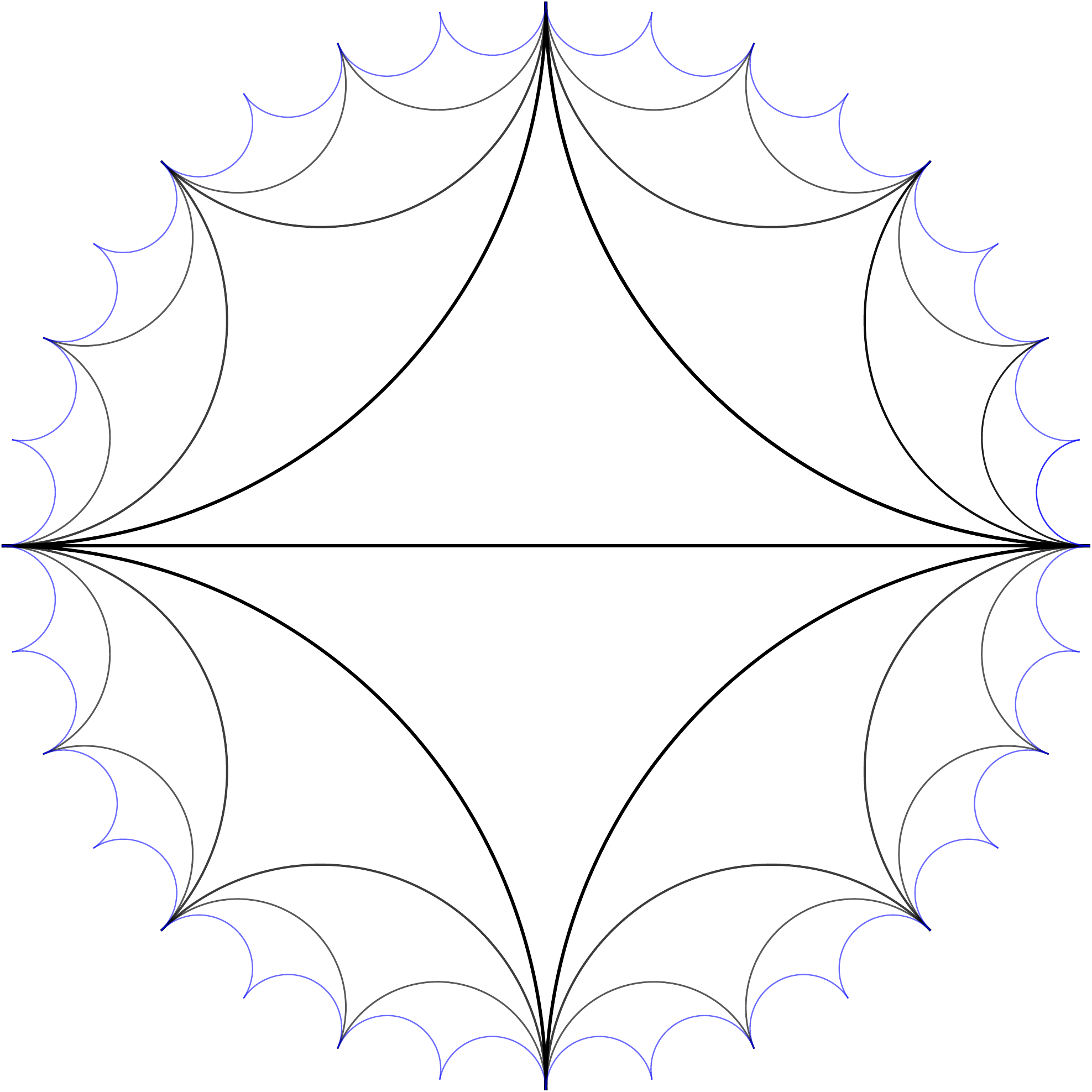}
    \caption{The Farey graph of order~$4$}
    \label{fig:FareyK}
\end{figure}

\begin{eg}\label{FareyExample}
The \emph{Farey graph $F_1$ of order~$1$} is obtained from a 4-cycle whose edges are coloured blue by adding a chord.
Recursively, the \emph{Farey graph $F_{k+1}$ of order $k+1$} is obtained from $F_k$ by adding a new vertex $v_e$ for each blue edge $e$ of~$F_k$, joining it to the two endvertices of~$e$ with blue edges, and uncolouring the previously blue edge~$e$; see \autoref{fig:FareyK}.
Now let $k\in\Nbb$ be any number and let us consider~$F_k$.

Each non-blue edge of~$F_k$ leaves two components after deleting its endvertices, and therefore defines a separation of~$F_k$ in the obvious way.
Let $N$ be the set of all separations of~$F_k$ defined in this way.
We claim that each separation in~$N$ forms an entanglement of its own.
To see that these singletons satisfy~\ref{Entangle1}, consider any separation $\{A,B\}\in N$, and let $\{C,D\}$ be any separation of~$G$ which crosses~$\{A,B\}$.
It~suffices to show that of every two corners lying on the same side of $\{A,B\}$, at least one has order larger than~$\vert A\cap B\vert=2$.
Since $A\cap B$ induces a~$K^2$, we may assume without loss of generality that $A\cap B\se C$.
By symmetry, it suffices to show that the corner $c:=\{A\cap C,B\cup D\}$ has order at least three.
If the separator of $c$ has size at most two, then it is equal to $A\cap B$, and $c=\{A,B\}$ follows because $A\sm B$ and $B\sm A$ are connected.
In particular, $A\cap C=A$ implies~$A\se C$, and $B\cup D=B$ implies $B\supseteq D$, so $\{A,B\}$ and $\{C,D\}$ are nested.
Since this would contradict our assumptions, $c$ must have order at least three, as desired.
Hence, each separation in~$N$ forms an entanglement in~$F_k$, so each separation in~$N$ is a \friendly\ separation of~$F_k$. 

The set of all separations of~$F_k$ whose separators span a~$K^2$ is nested and witnesses that there is no tangle in $F_k$ by  \cite[Theorem~12.5.1]{DiestelBook16noEE}.
\end{eg}

\begin{figure}[ht]
    \centering
    \includegraphics[height=5\baselineskip]{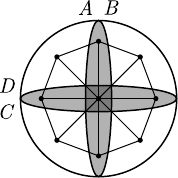}
    \caption{The situation in \autoref{eg:Wheel}}
    \label{fig:Wheel}
\end{figure}

\begin{eg}\label{eg:Wheel}
We claim that wheels have no entanglements.
Indeed, let $G$ be a wheel and let us suppose for a contradiction that there is an entanglement in~$G$.
Let $\{A,B\}$ be a separation of~$G$ that lies in an entanglement and whose side $A$ is inclusionwise minimal among all separations of~$G$ that lie in entanglements.
Since $A\sm B$ and $B\sm A$ are non-empty and the centre $c$ of the wheel is joined to all other vertices, $c$~can only be contained in $A\cap B$.
Furthermore, $\vert A\cap B\vert\ge 3$ since $G$ is 3-connected.
Pick any two vertices $a\in A\sm B$ and $b\in B\sm A$, and let $\{C,D\}$ be the separation of~$G$ with $C\cap D=\{a,b,c\}$.
Let $P$ and $Q$ be the two internally disjoint $a$--$b$ paths through the rim of the wheel.
Since $A\cap B$ meets both $P$ and~$Q$ in internal vertices, it follows that $\{A,B\}$ and $\{C,D\}$ cross and that all four corners have order at most~$\vert A\cap B\vert$.
Hence \ref{Entangle1} implies that at least one of the corners on the $A$-side of $\{A,B\}$ lies in an entanglement.
This contradicts the minimal choice of~$A$.
\end{eg}

\section{\Friendly\ separations are nested}\label{sec:OptimalNested}

For a finite graph~$G$ and a separation $s$ of~$G$, let us denote by $x(s)$ the number of separations in entanglements in~$G$ which are crossed by~$s$, and call $x(s)$ the \emph{crossing number} of~$s$ in~$G$.

\begin{lem}\label{GsimplyResolvable}
Let $G$ be any finite graph.
Suppose that for all entanglements $\eps_1,\eps_2$ in~$G$ \emph{(}possibly with $\eps_1=\eps_2$\emph{)} and any two crossing separations $s_1\in\eps_1$ and $s_2\in\eps_2$, there exist an index $i\in\{1,2\}$ and a separation $c\in\eps_i$ such that $x(c)<x(s_i)$.
Then the \friendly\ separations of~$G$ are nested.\qed
\end{lem}

\begin{lem}\label{AllFish}
Let $G$ be any graph, let $r,s$ be two crossing separations of~$G$, and let $c,d$ be two opposite corners of~$r,s$.
For every separation $t$ of~$G$ the following assertions hold:
\begin{enumerate}[label=\textnormal{(\roman*)}]
    \item\label{F1} If $t$ crosses at least one of $c$ and~$d$, then $t$ crosses at least one of $r$ and~$s$.
    \item\label{F2} If $t$ crosses both $c$ and~$d$, then $t$ crosses both $r$ and~$s$.
    \item\label{F3} Neither $r$ nor $s$ crosses $c$ or~$d$.
\end{enumerate}
\end{lem}
\begin{proof}
\ref{F1} holds by \cite[Lemma~12.5.5]{DiestelBook16noEE}, whose proof works for both finite and infinite graphs.
\ref{F2} is straightforward if one shows the contrapositive.
\ref{F3} is trivial.
\end{proof}

\begin{cor}\label{CornerCounting}
Let $G$ be any finite graph, let $r,s$ be two crossing separations in entanglements in~$G$, and let $c,d$ be two opposite corners of~$r,s$.
For every separation $t$ of~$G$ we have $x(c)+x(d)<x(r)+x(s)$.
\end{cor}
\begin{proof}
Combining \ref{F1}--\ref{F3} of \autoref{AllFish} gives $x(c)+x(d)\le x(r)+x(s)$.
Since $s$ and $r$ lie in entanglements and cross, they are counted in $x(r)$ and in $x(s)$, but they contribute to neither $x(c)$ nor $x(d)$ by \ref{F3}; hence the inequality is strict.
\end{proof}

\begin{lem}\label{GstandardArg}
Let $G$ be any finite graph.
Suppose that for all entanglements $\eps_1,\eps_2$ in~$G$ \emph{(}possibly with $\eps_1=\eps_2$\emph{)} and any two crossing separations $s_1\in\eps_1$ and $s_2\in\eps_2$, at least one of the following conditions is satisfied:
\begin{enumerate}[label=\textnormal{(C\arabic*)}]
    \item\label{GstandardArg1} there are opposite corners $c_1,c_2$ of~$s_1,s_2$ with $c_1\in\eps_1$ and $c_2\in\eps_2$;
    \item\label{GstandardArg2} two opposite corners of $s_1,s_2$ are in~$\eps_1$, and the other two opposite corners of $s_1,s_2$ are in~$\eps_2$.
\end{enumerate}
Then the \friendly\ separations of~$G$ are nested.
\end{lem}
\begin{proof}
It suffices to show that the premise of \autoref{GsimplyResolvable} is satisfied.
For this, let $\eps_1,\eps_2$ be any entanglements in~$G$ (possibly with $\eps_1=\eps_2$) and let $s_1,s_2$ be two crossing separations with $s_1\in \eps_1$ and~$s_2\in \eps_2$.

First, suppose that by~\ref{GstandardArg1} there are opposite corners $c_1,c_2$ of $s_1,s_2$ with $c_1\in \eps_1$ and $c_2\in \eps_2$.
By \autoref{CornerCounting}, we have $x(c_1)+x(c_2)< x(s_1)+x(s_2)$.
An indirect proof finds an $i\in \{1,2\}$ with $x(c_i)<x(s_i)$.

Second, suppose that by~\ref{GstandardArg2} there are two opposite corners $c_1,c_1'$ of $s_1,s_2$ are in~$\eps_1$, and the other two opposite corners $c_2,c_2'$ of $s_1,s_2$ are in~$\eps_2$.
By \autoref{CornerCounting}, we have $x(c_i)+x(c_i')< x(s_1)+x(s_2)$ for both~$i=1,2$.
Without loss of generality, we have $x(s_1)\le x(s_2)$.
Hence either $x(c_2)<x(s_2)$ or $x(c_2')<x(s_2)$.
\end{proof}

Let us write $|s|:=|A\cap B|$ for a separation~$s=\{A,B\}.$
If two separations $s_1$ and $s_2$ of a graph~$G$ cross and $c_1,c_2$ are two opposite corners of $s_1,s_2$, then the orders of these corners sum to $\vert c_1\vert+\vert c_2\vert=\vert s_1\vert+\vert s_2\vert$.
The important part of this equality is the inequality $\vert c_1\vert+\vert c_2\vert\le \vert s_1\vert+\vert s_2\vert$, which is known as \emph{submodularity}, and which is the only part of the equality that we will need in the proofs.

\begin{thm}\label{GEntanglementDisplay}
The \friendly\ separations of any finite graph are nested.
\end{thm}

\begin{proof}
It suffices to show that the premise of \autoref{GstandardArg} is satisfied.
For this, let $\eps_1$ and $\eps_2$ be any entanglements in~$G$, possibly with $\eps_1=\eps_2$, and let $s_1\in \eps_1$ and $s_2\in \eps_2$ be two crossing separations.
Without loss of generality, we have $\vert s_1\vert\le \vert s_2\vert$.
Let us colour a corner of $s_1,s_2$ green if it has order at most~$\vert s_2\vert$.

\begin{sublem}\label{atLeastThreeGreen}
At least three corners of $s_1,s_2$ are green.
\end{sublem}
\begin{proof}
Suppose for a contradiction that at most two corners of $s_1,s_2$ are green.
By submodularity and since~$\vert s_1\vert\le\vert s_2\vert$, at least one of any two opposite corners must be green.
So there are exactly two green corners, and since they cannot be opposite they must be adjacent.
As the remaining two corners are not green by assumption, they have order greater than~$\vert s_2\vert$.
By submodularity, this means that the green corners in fact have order less than~$\vert s_1\vert$.
Then either $\eps_1$ or $\eps_2$ contains a green corner by \ref{Entangle1}.
But then this green corner has order equal to $\vert s_1\vert$ or $\vert s_2\vert$ by~\ref{Entangle1}, contradicting our observation that it has order less than~$\vert s_1\vert$ and~$\vert s_2\vert$.
\end{proof}

By \autoref{atLeastThreeGreen}, at least three corners of $s_1,s_2$ are green.
Hence it suffices to consider the following two cases.
See \autoref{fig:Case1} for a depiction of Case~1.

\begin{figure}[ht]
    \centering
    \includegraphics[height=5\baselineskip]{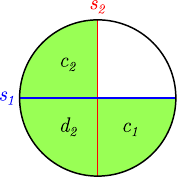}
    \caption{The situation in Case~1}
    \label{fig:Case1}
\end{figure}

\textbf{Case~1:} In the first case, precisely three corners of $s_1,s_2$ are green.
Then two green corners $c_2,d_2$ lie on the same side of $s_2$, so at least one of them is contained in~$\eps_2$ by~\ref{Entangle1}, say $c_2\in \eps_2$.
Hence $c_2$ has order exactly~$\vert s_2\vert$.
So the corner $c_1$ opposite of $c_2$ has order at most~$\vert s_1\vert$ by submodularity; in particular, $c_1$ is green.
Note that $c_1$ and~$d_2$ lie on the same side of~$s_1$.
The corner opposite of~$d_2$ is not green, so has order more than~$\vert s_2\vert$.
Hence $d_2$ has order less than~$\vert s_1\vert$ by submodularity.
So by \ref{Entangle1}, at least one of $d_2$ and~$c_1$ is contained in~$\eps_1$ and has order equal to~$\vert s_1\vert$.
This can only be~$c_1$.
So $c_1,c_2$ are opposite corners of $s_1,s_2$ with $c_1\in \eps_1$ and $c_2\in \eps_2$, giving~\ref{GstandardArg1}.

\textbf{Case~2:} In the second case, all four corners are green.
Applying \ref{Entangle1} on both sides of~$s_2\in \eps_2$, we find corners $c_2,c_2'$ of $s_1,s_2$ with $c_2,c_2'\in \eps_2$ such that $c_2,c_2'$ do not lie on the same side of~$s_2$.
Moreover, $c_2$ and $c_2'$ have order exactly~$\vert s_2\vert$ by~\ref{Entangle1}.
We consider two subcases.

\textbf{Subcase 2A:} In the first subcase, the two corners $c_2,c_2'$ are adjacent, so they lie on the same side of~$s_1$.
Let $c_1$ be the corner opposite of~$c_2$, and let $c_1'$ be the corner opposite of~$c_2'$.
The corners $c_1,c_1'$ have order at most~$\vert s_1\vert$ by submodularity.
Moreover, $c_1$~and~$c_1'$ lie on the same side of~$s_1$.
Hence at least one of $c_1$ and~$c_1'$ is contained in~$\eps_1$ by~\ref{Entangle1}, and we already know that its opposite corner is contained in~$\eps_2$, giving~\ref{GstandardArg1}.

\textbf{Subcase 2B:} In the second subcase, the two corners $c_2,c_2'$ are opposite.
Since $c_2$ and $c_2'$ have order~$\vert s_2\vert$, submodularity with $\vert s_1\vert\le\vert s_2\vert$ implies \mbox{$\vert s_1\vert=\vert s_2\vert$}.
Therefore, by symmetry we can repeat the entire argumentation up to this point with the roles of $s_1$~and~$s_2$ interchanged to find two opposite corners $c_1,c_1'$ of $s_1,s_2$ with $c_1,c_1'\in \eps_1$.
If the sets $\{c_1,c_1'\}$ and $\{c_2,c_2'\}$ intersect, then they are equal, so $c_2$ and $c_2'$ are opposite corners of $s_1,s_2$ with $c_2\in \eps_1$ and $c_2'\in \eps_2$, giving~\ref{GstandardArg1}.
Otherwise, $\{c_1,c_1'\}$ and $\{c_2,c_2'\}$ are disjoint, and then $c_1,c_1'$ are two opposite corners of $s_1,s_2$ in~$\eps_1$ while $c_2,c_2'$ are the other two opposite corners of $s_1,s_2$ and are in~$\eps_2$, giving~\ref{GstandardArg2}.
\end{proof}

\begin{proof}[Proof of \autoref{Main}]
Let $G$ be a finite graph, and let $N$ denote its set of friendly separations.
The set $N$ is nested by \autoref{GEntanglementDisplay}, and it efficiently distinguishes all the tangles in~$G$ by \autoref{TanglesInduceEntanglements}.
As is well-known~\cite[(9.1)]{GMX}, $N$ defines a tree-decomposition $\Tcal$ of~$G$, which efficiently distinguishes all the tangles in~$G$ since $N$ does.
\end{proof}

\begin{rem}
To construct the tree-decomposition $\cT$ that efficiently distinguishes all the tangles in the proof of \autoref{Main}, we have used all entanglements in~$G$ (to first define $N$ and then~$\Tcal$), not just the ones induced by the pairs of distinguishable tangles.
It is possible to adjust the entire framework of this section to only work with the set $\Ecal$ of tangle-induced entanglements instead, to obtain a nested set~$N'=N'(\Ecal)$, which may be incomparable with~$N$ (as set), and then obtain a tree-decomposition from~$N'$; we do this in more detail in \autoref{TechnicalMain} (because there we must restrict to a subset of all the entanglements).
This would make sure that every separation in $N'$ (and hence of the tree-decomposition) efficiently distinguishes two tangles in~$G$.
However, there is an alternative way to achieve this: we can consider the subset $N''\se N$ formed by the separations that efficiently distinguish some two tangles, and then consider the tree-decomposition defined by~$N''$.
\end{rem}

\section{Entanglements in locally-finite infinite graphs}\label{sec:Infinite}

Recall that a graph is \emph{locally finite} if each of its vertices has only finitely many neighbours.
In this section, we extend \autoref{GEntanglementDisplay} to locally-finite infinite graphs.
The proof of \autoref{GEntanglementDisplay} almost works for locally-finite infinite graphs.
The only places where we use finiteness are where we use the crossing numbers~$x(s)$; indeed, we only need that all relevant crossing numbers are finite.
To ensure this, we combine local finiteness with two other customary conditions, \emph{tightness} and \emph{finite boundedness}; see \autoref{finiteKcrossingNumber}.
Then we extend \autoref{GEntanglementDisplay} to infinite graphs under the combination of the three conditions.
The combination of the three conditions is mild in the sense that the extension result, \autoref{TechnicalMain}, is strong enough for its application in~\cite{GraphDec}.

A separation $\{A,B\}$ of a graph $G$ is \emph{tight} if there are components $C_A$ and $C_B$ of $G-(A\cap B)$ with $C_A\se G[A\sm B]$ and $C_B\se G[B\sm A]$ such that $N_G(C_A)=A\cap B=N_G(C_B)$.
An entanglement in a graph is \emph{tight} if it consists of tight separations.
For instance, entanglements induced by pairs of tangles are tight \cite[Lemma~6.1]{InfiniteSplinters}.

\begin{lem}\label{finiteKcrossingNumber}
Let $G$ be any locally finite connected graph and~$k\in\Nbb$.
Then every tight finite-order separation of $G$ is crossed by only finitely many tight separations of~$G$ of order at most~$k$.
\end{lem}
\begin{proof}
This fact is well-known; see e.g.\ the proof of \cite[Proposition~6.2]{InfiniteSplinters}.
\end{proof}

A set $\Ecal$ of entanglements is \emph{finitely bounded} if there is $k\in\Nbb$ with $\vert s\vert \le k$ for all~$s\in\bigcup\Ecal$.
Let $G$ be any graph, and let $\Ecal$ be a set of entanglements in~$G$. 
Suppose that $\Ecal$ is finitely bounded.
If $G$ is locally finite but infinite, we additionally assume that all entanglements in~$\Ecal$ are tight, so that each separation in~$\bigcup\Ecal$ crosses only finitely many separations in~$\bigcup\Ecal$ by \autoref{finiteKcrossingNumber}.
A separation $\{A,B\}$ in an entanglement $\eps\in\Ecal$ is \emph{$\Ecal$-\friendly } if no other separation in~$\eps$ crosses less separations in~$\bigcup\Ecal$.

\begin{thm}\label{TechnicalMain}
Let $G$ be any locally-finite connected graph and let $\Ecal$ be any finitely bounded set of tight entanglements in~$G$.
Then the set of $\Ecal$-\friendly\ separations of~$G$ is nested.
\end{thm}
\begin{proof}
The plan is to walk through \autoref{sec:OptimalNested} once more and see that everything adjusts to and works in the setting of the theorem.
First, we adjust the crossing numbers: $x(s)$ counts only the separations in entanglements in $\Ecal$ that cross~$s$.
Then $x(s)$ is finite for all $s\in\bigcup\Ecal$, by \autoref{finiteKcrossingNumber}.

In \autoref{GsimplyResolvable}, we only consider entanglements in~$\Ecal$, and use that the crossing-numbers $x(s_i)$ are finite.
\autoref{AllFish} is stated and proved for arbitrary graphs.
In \autoref{CornerCounting}, we only consider entanglements in~$\Ecal$, so $x(r)$ and $x(s)$ are finite; then the proof extends.
\autoref{GstandardArg} extends similarly, and so does the proof of \autoref{GEntanglementDisplay}.
\end{proof}

Recall that every end of a graph induces a tangle of infinite order; in particular, every pair of ends induces an entanglement.
Two ends of a graph are \emph{$({<}\,k)$-distinguishable} (for $k\in\Nbb$) if their induced tangles are distinguished by a separation of order less than~$k$.

\begin{cor}
Let $G$ be any locally-finite connected graph and $k\in\Nbb$.
Let $\Ecal$ be the set of all entanglements in~$G$ that are induced by pairs of $({<}\,k)$-distinguishable ends of~$G$.
Then the set of $\Ecal$-friendly separations of~$G$ is nested and efficiently distinguishes every pair of $({<}\,k)$-distinguishable ends of~$G$.\qed
\end{cor}

R\"uhmann showed a result that is somewhat similar to the above corollary, see~\cite[Theorem 6.1.6]{Tim}.
For more on infinite trees of tangles, we refer to~\cite{CanonicalTreesofTDs,carmesinhalinconj,InfiniteSplinters,ToTinfOrder,jacobs2023efficiently}.

\section{Abstract entanglements}\label{sec:AbstractEntangle}

In this section, we introduce an abstract setting which is more general than separations of graphs, and generalise \autoref{Main} to this abstract setting.

A \emph{separation} is a set of the form~$\{A,B\}$ with~$A\neq B$.
We refer to $A$ and $B$ as the \emph{(opposite) sides} of~$\{A,B\}$.
An \emph{uncrossing-setting} on a set~$S$ of separations is a pair $(S,{\sim})$ where $\sim$ is an anti-reflexive symmetric binary-relation on~$S$.
Instead of writing $r\sim s$ we say that $r$ and~$s$ \emph{cross}, and any two elements of~$S$ that do not cross are \emph{nested}.
A set of separations in~$S$ is \emph{nested} if its elements are pairwise nested.

A \emph{corner-map} for an uncrossing-setting $(S,{\sim})$ is a map $\boxplus$ which assigns to every unordered pair of crossing separations $r=\{A,B\}$ and $s=\{C,D\}$ four pairwise distinct separations $L_{\{r,s\}}(\{X,Y\})$, one for each choice of sides $X\in\{A,B\}$ and $Y\in\{C,D\}$, subject to condition \ref{BetterCorner} below.
We allow any number of these corners to be elements of~$S$, but we do not require them to be elements of~$S$.
A corner of $r,s$ that is contained in~$S$ shall be called an \emph{$S$-corner} for emphasis.
As $r$ and $s$ will always be clear from context, we reduce the notation $L_{\{r,s\}}(\{X,Y\})$ to $L(X,Y)$ for convenience.

\begin{eg}
If two separations $\{A,B\}$ and $\{C,D\}$ of a graph~$G$ cross, then the four corners are the usual corners $L(X,Y):=\{X\cap Y,X'\cup Y'\}$ for $\{X,X'\}=\{A,B\}$ and $\{Y,Y'\}=\{C,D\}$.
\end{eg}

Two distinct corners $L(X,Y)$ and $L(X',Y')$ are \emph{opposite} if $X,X'$ are opposite sides of~$\{A,B\}$ and $Y,Y'$ are opposite sides of~$\{C,D\}$.
They are \emph{adjacent} if they are not opposite, which is equivalent to having $X=X'$ or~$Y=Y'$ but not both.
They \emph{lie on the same side} of $\{A,B\}$ if $X=X'$, and similarly they \emph{lie on the same side} of $\{C,D\}$ if~$Y=Y'$.
Note that distinct corners that lie on the same side of~$r$ or of~$s$ are adjacent.
Condition~\ref{BetterCorner} generalises \autoref{CornerCounting} and reads as follows:

\begin{enumerate}[label={\textnormal{(F)}}]
    \item\label{BetterCorner} Every two opposite $S$-corners $c,d$ of~$r,s$ satisfy the following three conditions.
    \begin{enumerate}[label={\textnormal{(F\arabic*})}]
    	 \item\label{F1S} If $t\in S$ crosses at least one of $c$ and~$d$, then $t$ crosses at least one of $r$ and~$s$.
    	 \item\label{F2S} If $t\in S$ crosses both $c$ and~$d$, then $t$ crosses both $r$ and~$s$.
    	 \item\label{F3S} Neither $r$ nor $s$ crosses $c$ or~$d$.
    \end{enumerate}
\end{enumerate}

An \emph{order-function} is a map \[\vert\cdot\vert\colon S\cup\{\text{corners of crossing separations in }S\}\to\Rbb_{\ge 0}.\]
Then $\vert s\vert$ is the \emph{order} of~$s$.
An order-function $\vert\cdot\vert$ is \emph{submodular} if for every two crossing elements $r,s\in S$ and opposite corners $c,d$ of~$r,s$ it satifies $\vert c\vert+\vert d\vert\le \vert r\vert+\vert s\vert$.
A \emph{submodular uncrossing-setting} on a set~$S$ of separations is a triple $(S,{\sim},\boxplus,\vert\cdot\vert)$ formed by an uncrossing-setting $(S,{\sim})$ with a corner-map~$\boxplus$ and a submodular order-function~$\vert\cdot\vert$.

An \emph{entanglement} in a submodular uncrossing-setting on a set~$S$ of separations is a non-empty subset $\eps\se S$ which exhibits the following property:
\begin{enumerate}[label=($\Ecal$)]
    \item If a separation $r\in \eps$ is crossed by an~$s\in S$ so that two adjacent corners on the same side of~$r$ have order at most~$\vert r\vert$, then at least one of these two corners has order equal to~$\vert r\vert$ and lies in~$\eps$.
\end{enumerate}

Suppose now that $S$ is finite.
For every $s\in S$ we denote by $x(s)$ the number of separations in entanglements which are crossed by~$s$, and we call $x(s)$ the \emph{crossing-number} of~$s$.

\begin{lem}\label{ScrossingSubmod}
Let $r,s$ be two crossing separations in entanglements in a submodular uncrossing-setting on a set~$S$ of separations.
Then for every two opposite $S$-corners $c,d$ of $r,s$ we have $x(c)+x(d)<x(r)+x(s)$.
\end{lem}
\begin{proof}
This follows from \ref{F1S}--\ref{F3S} just like in the proof of \autoref{CornerCounting}.
\end{proof}

A separation $s\in S$ is \emph{\friendly } if it occurs in an entanglement~$\eps$ and no other separation in~$\eps$ crosses less separations in entanglements.

\begin{thm}\label{AbstractEntanglementDisplay}
The \friendly\ separations in a finite submodular uncrossing-setting are nested.
\end{thm}

\begin{proof}
The proof is analogous to the proof of \autoref{GEntanglementDisplay}, including \autoref{GsimplyResolvable} and \autoref{GstandardArg}, with just one exception:
instead of \autoref{CornerCounting}, we use~\autoref{ScrossingSubmod}.
\end{proof}

\autoref{AbstractEntanglementDisplay} clearly implies \autoref{Main}, and it yields the following version of \autoref{Main} for matroids.
We state the theorem using the terminology of \cite[§4.2]{ProfilesNew}.
The usual order-function for matroid-separations is well known to be submodular, see e.g.~\cite{ProfilesNew}. Matroid-separations exhibit \ref{BetterCorner}: indeed, the proof of \autoref{AllFish} extends to matroid-separations verbatim.
Hence matroid-separations form a submodular uncrossing-setting.

\begin{thm}
For every finite matroid~$M$, the set of \friendly\ separations of~$M$ is a nested set of separations; and hence gives rise to a tree-decomposition distinguishing all tangles efficiently.\qed
\end{thm}

\paragraph{\bf Concluding remarks}
In \cite{Local2seps}, $r$-local 2-separations of graphs have been introduced, which need not separate the graph globally but which separate it $r$-locally in that they separate a ball of radius $r/2$ around their separators.
While it is not obvious how the notion of tangles could be generalised to $r$-local separations, this can be achieved for entanglements with a slightly different notion of $r$-local separations, as announced in~\cite{Local2seps}.
We would also like to mention that \autoref{TechnicalMain} and \autoref{AbstractEntanglementDisplay} will be used in upcoming work to find graph-decompositions, see for example~\cite{GraphDec}.

\subsection*{Acknowledgement}

We thank two referees for valuable comments that greatly improved this note.
One comment fixed a critical error in the setup for abstract entanglements, and we are grateful to the referee for spotting and fixing it.
We thank Sandra Albrechtsen for pointing out and fixing an error in \autoref{CornerCounting}.
We are grateful to Raphael W.~Jacobs and Paul Knappe for feedback on a very early~draft.
We thank Nathan Bowler for bringing the work of Rühmann~\cite{Tim} to our attention.

\bibliographystyle{plain}
\bibliography{collective}
\end{document}